\documentclass[10pt]{article}

\usepackage{amssymb}
\usepackage{amsfonts}
\usepackage{amsmath}
\usepackage{amsthm}
\input xy
\xyoption{all}

\date{}

\newtheorem{proposition}{Proposition}[section]
\newtheorem{theorem}[proposition]{Theorem}
\newtheorem{lemma}[proposition]{Lemma}
\newtheorem{example}[proposition]{Example}
\newtheorem{definition}[proposition]{Definition}

\newtheorem{remark}{Remark}

\begin{document}

\author{H. Melis Tekin Akcin}

\title{An analogue of Amitsur's property for the ring of pseudo-differential operators}


\maketitle

\begin{abstract}
Let $R$ be a ring with a derivation $\delta$. In this paper, we investigate when the left T-nilpotent radideal satisfies an analogue of Amitsur's property for pseudo-differential operator rings $R((x^{-1};\delta))$. We also obtain an alternative characterization of the prime radical of $R((x^{-1};\delta))$.


\noindent

{\em Key Words: Amitsur's property, $\delta$-compatible rings, left T-nilpotent, prime radical, pseudo-diﬀerential operator ring.}

 {\em Mathematics subject classification
2010: 16N40, 16N60, 16S32.}
\end{abstract}


\section{Introduction}

In this paper, we study rings of pseudo-differential operators, which can be seen as noncommutative generalizations of commutative Laurent series rings. 
The idea of using the algebra of pseudo-differential operators $R((\delta^{-1}))$ is started with Schur (see \cite{Schur}), later works on these algebras have done by Goodearl \cite{Goodearl1} and Tuganbaev \cite{Tuganbaev-Pseudo}. In \cite{Tuganbaev-Pseudo}, Tuganbaev has studied the ring theoretical properties of pseudo-differential operator rings.
Besides being used to construct new examples in ring theory, these rings also have some applications in different fields of mathematics, see \cite{Gelfand} and \cite{Sato} for more information.

Throughout this paper, $R$ denotes an associative ring with identity (unless otherwise stated), an ideal means a two-sided ideal and the notation "$\leq$" is used to denote the ideals. Let $R$ be a ring equipped with a derivation $\delta$ (i.e., $\delta$ is an additive map on $R$ satisfying the product rule $\delta(ab)=\delta(a)b+a\delta(b)$, for each $a, b \in R$).
The pseudo-differential operator ring over the coefficient ring $R$ formed by formal series $\sum_{i=-\infty}^{n}a_{i}x^{i}$, where $x$ is a variable, $n$ is an integer (maybe negative), the coefficients $a_{i}$ belong to the ring $R$ and is denoted by the notation $R((x^{-1};\delta))$.
In \cite[Proposition 7.2]{Tuganbaev-Pseudo}, it is verified that $R((x^{-1};\delta))$ satisfies all the ring axioms, where the 
addition is defined as usual and multiplication is defined with respect to the relations
\begin{center}
$xa=ax+\delta(a)$, $x^{-1}a=\sum_{i=0}^{\infty}(-1)^{i}\delta^{i}(a)x^{-i-1}$,
\end{center}  
for all $a\in R$.
If $\delta$ is the zero derivation, then there exists an isomorphism of the ring $R((x^{-1};\delta))$ onto the ordinary Laurent series ring $R((x))$ (This isomorphism maps $x^{-1}$ onto $x$).
The {\em Amitsur's property} of a radical says that
the radical of a polynomial ring is again a polynomial ring. This nomenclature is used since it was Amitsur who initially proved that many classical radicals such as the prime, Levitzki, Jacobson, and Brown–McCoy have this property.  Moreover, in \cite[Proposition 4.10]{Hong1}, it is proved that the left T-nilpotent radideal of a polynomial ring also satisfies the Amitsur's property. It is a natural question to extend the Amitsur's property
for other ring extensions. As a generalization of Amitsur's property, in \cite{Hong4}, the concept of {\em $\delta$-Amitsur property} is introduced for the ring of differential operators. In \cite{Ferrero}, Ferrero, Kishimoto and Motose have proved that the Jacobson, prime and Wedderburn radicals again possess $\delta$-Amitsur's property. Also, in \cite[Theorem 3.3]{Hong4}, it is showed that the left T-nilpotent radideal of the ring of differential operators satisfies the $\delta$-Amitsur property. 
 
In their seminal papers \cite{Hong1} and \cite{Hong4}, the authors have studied on characterization of the left T-nilpotent radideals of skew Laurent polynomial rings and 
the rings of differential operators.
Our primary motivation in this paper is to give a description of the left T-nilpotent radideals   
of pseudo-differential operator rings. 
Before proceeding the main results, we need to recall some concepts and definitions which will be useful while discussing the left T-nilpotent radideals of pseudo-differential operator rings.

Let $R$ be a ring and $\delta$ be a derivation of $R$,
we say that a subset $S\subseteq R$ is a {\em $\delta$-subset} if $\delta(S)\subseteq S$.
Let $I$ be an ideal of $R$. If $I$ is a $\delta$-subset of $R$, then $I$ is called a 
{\em $\delta$-ideal} of $R$. According to $\cite{Hashemi2}$, an ideal $I$ is called a {\em $\delta$-compatible
ideal} if for each $a, b\in R$, $ab\in I$ implies $a\delta(b)\in I$. If the zero ideal is $\delta$-compatible, then the ring $R$ is called a $\delta$-compatible ring. \\ 
Let $R$ be a ring with a derivation $\delta$ and if $I$ is a $\delta$-ideal of $R$, then 
\begin{center}
$\bar{\delta}:R/I \longrightarrow R/I$
\end{center} 
is a derivation of $R/I$ induced by the derivation $\delta$.

\begin{lemma}\cite[Lemma 2.1]{Hashemi2}\label{deltacompatible}
Let $R$ be a ring and $\delta$ be a derivation of $R$. Assume that $R$ is $\delta$-compatible.
If $ab=0$, then $a\delta^{n}(b)=\delta^{m}(a)b=0$ for any non-negative integers $n,m$.
\end{lemma}
\noindent
If $S$ is a subset of a ring $R$, we denote the {\em left annihilator} of $S$ in $R$ by the notation  $(0:S)$.
\noindent
For an arbitrary ring $R$, the ideals $R^{(\alpha)}$ are defined recursively in \cite{Gardner} as follows:
$R^{(0)}=0$, $R^{(\alpha+1)}/R^{(\alpha)}=(0:R/R^{(\alpha)})$ and $R^{(\alpha)}=\underset{\beta<\alpha}{\cup}R^{(\beta)}$, if $\alpha$ is a limit ordinal. If $R^{(\mu)}=R$ for some ordinal $\mu$, then the series
\begin{center}
$0=R^{(0)}\subseteq R^{(1)} \subseteq \ldots \subseteq R^{(\alpha)}\subseteq \ldots \subseteq R^{(\mu)}=R$
\end{center}
is called the {\em upper left annihilator series} of $R$.
\begin{remark}\label{Ralpha}
It can be seen easily that, by using transfinite induction \cite[Proposition 9, Section 1.3]{Divinsky}, 
the ideals defined as above are actually $\delta$-ideals,  i.e., 
$\delta(R^{(\alpha)})\subseteq R^{(\alpha)}$ for each ordinal $\alpha$. 
\end{remark}


\section{Main Results}

In radical theory, it is interesting to characterize the radicals of 
ring extensions in terms of the base rings. 
In \cite{Hong1} and \cite{Hong4}, the authors have proved the analogue 
of this question for the left T-nilpotent radideals of skew Laurent polynomial
rings and the ring of differential operators, by using K\"{o}nig's tree lemma.  
One difficulty with extending the situation for pseudo-differential operator rings
is that we no longer have a finite coefficient set. In this section, we study the left T- nilpotent radideal of pseudo-differential operator rings $R((x^{-1};\delta))$. 
We begin by giving the concept of left T-nilpotent set and its properties.

\begin{definition}\cite{Lam}
A set $S\subseteq R$ is called left T-nilpotent if for any countable sequence of elements 
$s_1,s_2,\ldots \in S$, there exists an integer $k\geq 1$ such that 
\begin{center}
$s_1s_2\ldots s_k=0$.
\end{center}
\end{definition}
\noindent
Note that right T-nilpotent sets are defined in a similar way and we say that a set is 
{\em T-nilpotent}, if it is both left and right T-nilpotent.
The terms are due to Bass \cite{Bass}, but the concepts were introduced by Levitzki \cite{Levitzki}. 
By the very definition, it is easy to see that any subset 
of a left T-nilpotent set is 
again left T-nilpotent. Also, if an ideal is left T-nilpotent, then it is nil.
Moreover, if an ideal is nilpotent, then it is left T-nilpotent.
\begin{proposition}\cite[Proposition 23.15]{Lam}\label{Prop2.2} 
Let $R$ be a ring and $I$ be an ideal of $R$. If $I$ is left T-nilpotent, then $I \subseteq P(R)$ where $P(R)$ denotes the prime radical of $R$. In particular, $I$ is locally nilpotent.   
\end{proposition}
\begin{lemma}\cite[Lemma 4.2]{Hong1}
Let $R$ be a ring, $I\subseteq R$ and $J$ be an ideal of $R$. If $I$ and $J$ are left T-nilpotent, 
then so is $I+J$.
\end{lemma}
\begin{proposition}\cite[Proposition 4.3]{Hong1} \label{2-sided}
Let $R$ be a ring, $I\subseteq R$ and $J$ be a one-sided ideal of $R$.
\begin{itemize}
\item[(1)] If $J$ is left T-nilpotent, then $RJR$ is left T-nilpotent.
\item[(2)] If $I$ and $J$ are left T-nilpotent, then $I+J$ is left T-nilpotent.
\end{itemize}
\end{proposition}
\noindent
For a deeper knowledge and basic results about left T-nilpotency, see \cite[section 4]{Hong1} and \cite[section 23]{Lam}.
The left T-nilpotent radideal of $R$ is denoted by $\mathcal{I}_{l}$ and defined as the ideal function given by
\begin{center}
$\mathcal{I}_{l}(R)=\sum\{I\leq R: \ I \text{ is left T-nilpotent} \}$.
\end{center}
As one might expect, the ideal $\mathcal{I}_{l}(R)$ does not need to be left T-nilpotent itself. 
\noindent
 
Let $I((x^{-1};\delta))$ be the subset of $R((x^{-1};\delta))$ whose coefficients are all contained in $I$. 
We begin with the following lemma which gives the relations between the ideals of $R((x^{-1};\delta))$ and $R$. 
\begin{lemma} \label{Ideal}
Let $R$ be a ring and $\delta$ be a derivation of $R$. Then the following statements hold:
\begin{itemize}
\item[(1)] If $I$ is a right ideal of $R$, then $I((x^{-1};\delta))$ is a right ideal of $R((x^{-1};\delta))$.
\item[(2)] Let $I$ be an ideal of $R$. Then $I((x^{-1};\delta))$ is an ideal of $R((x^{-1};\delta))$ if and only if $I$ is a $\delta$-ideal of $R$.  
\item[(3)] If $I$ is a nilpotent $\delta$-ideal of $R$, then $I((x^{-1};\delta))$ is a nilpotent ideal of $R((x^{-1};\delta))$.
\end{itemize}
\end{lemma}

\begin{proof}
The proof can be seen easily by using \cite[Lemma 2.1]{Paykan_2019}.
\end{proof}
\begin{lemma} \label{Tnilpotence}
Let $R$ be a ring and $\delta$ be a derivation of $R$. Assume that $R$
is $\delta$-compatible. If $aR$ is left T-nilpotent, then $\delta^{i}(a)R$ is left T-nilpotent for each non-negative integer $i$.
\end{lemma}  
\begin{proof}
Fix an arbitrary sequence of elements $r_{1},r_{2},\ldots \in R$. By the assumption, there exists an integer $k\geq 1$ such that $ar_{1}ar_{2}\ldots ar_{k}=0$.
Since $R$ is $\delta$-compatible, by Lemma \ref{deltacompatible}, we have $\delta^{i}(a)r_{1}\delta^{i}(a)r_{2}\ldots \delta^{i}(a)r_{k}=0$ for each non-negative integer $i$. This means that $\delta^{i}(a)R$ is left T-nilpotent.
\end{proof} 
\noindent

The following ring-theoretic characterization on left T-nilpotent rings is obtained by Levitzki \cite{Levitzki}. We state this result without the proof and interested reader is referred to see \cite{Levitzki} and \cite[Theorem 1.3]{Gardner} for more information.
\begin{theorem}
Let $R$ be a ring (maybe without identity). Then $R$ is left T-nilpotent if and only if the upper left 
annihilator series of $R$ exists.
\end{theorem}
This result enables us to obtain an analogue of Amitsur's property for the left T-nilpotent radideal of pseudo-differential operator rings.
\begin{lemma}\label{extofTnilpotence}
Let $R$ be a ring and $\delta$ be a derivation of $R$. If $I$ is a left T-nilpotent 
$\delta$-ideal of $R$, then $I((x^{-1};\delta))$ is a left T-nilpotent ideal of $R((x^{-1};\delta))$.
\end{lemma}

\begin{proof}
We will use Levitzki's characterization to prove that $I((x^{-1};\delta))$ is left T-nilpotent. Since $I$ is left T-nilpotent, the upper left annihilator series of $I$ exists.
Let
\begin{center}
$0=I^{(0)}\subseteq I^{(1)} \subseteq \ldots \subseteq I^{(\alpha)}\subseteq \ldots \subseteq I^{(\mu)}=I,$
\end{center}
where $I^{(\alpha+1)}/I^{(\alpha)}=(0:I/I^{(\alpha)})$ and $I^{(\alpha)}=\underset{\beta<\alpha}{\cup}I^{(\beta)}$, if $\alpha$ is a limit ordinal. 
We wish to obtain the upper left annihilator series for $I((x^{-1};\delta))$. By Remark \ref{Ralpha} and Lemma \ref{Ideal}(2), we have that $I^{(\alpha)}((x^{-1};\delta))$ is an ideal of $I((x^{-1};\delta))$ for any ordinal $\alpha$. 
Let $f(x)=\sum_{i=-\infty}^{n}a_{i}x^{i}\in I^{(\alpha+1)}((x^{-1};\delta))$, then for any 
$g(x)=\sum_{i=-\infty}^{m}b_{i}x^{i}\in I((x^{-1};\delta))$ we have that each coefficient
of the product $f(x)g(x)$ is a $\mathbb{Z}$-linear combination of terms of the form
\begin{center}
$a_{i_1}\delta^{k}(b_{j_1})$,
\end{center} 
where $a_{i_1}$ is any coefficient of $f(x)$, $b_{j_1}$ is any coefficient of $g(x)$ and $k$ is a non-negative integer.
Since $I$ is a $\delta$-ideal and $a_{i_1}\in I^{(\alpha+1)}$, by the construction of the upper left annihilator series we obtain
$a_{i_1}\delta^{k}(b_{j_1})\in I^{(\alpha)}$ for each $a_{i_1}$, $b_{j_1}$ and non-negative integer $k$. This means that $f(x)g(x)\in I^{(\alpha)}((x^{-1};\delta))$. 
If $\alpha$ is a limit ordinal, then we have 
\begin{center}
$I^{(\alpha)}((x^{-1};\delta))=\big(\underset{\beta<\alpha}{\cup}I^{(\beta)} \big)((x^{-1};\delta))=\underset{\beta<\alpha}{\cup}I^{(\beta)}((x^{-1};\delta))$.
\end{center}
Conversely, assume that $f(x)=\sum_{i=-\infty}^{n}a_{i}x^{i}\in I((x^{-1};\delta))$ such 
that 
\begin{center}
$f(x)I((x^{-1};\delta))\subseteq I^{(\alpha)}((x^{-1};\delta))$. 
\end{center}
We need to show that $f(x)\in I^{(\alpha+1)}((x^{-1};\delta))$. By the assumption, we have
$f(x)a\in I^{(\alpha)}((x^{-1};\delta))$ for each $a\in I$.
Since the leading term of $f(x)a$ is $a_{n}a$, we get that $a_{n}a\in I^{(\alpha)}$ for each 
$a\in I$. So, we obtain $a_{n}\in I^{(\alpha+1)}$. Now set $f'(x)=f(x)-a_{n}x^{n}$. 
By using the assumption and the fact that $a_{n}\in I^{(\alpha+1)}$, we get
\begin{center}
$f'(x)I((x^{-1};\delta))\subseteq I^{(\alpha)}((x^{-1};\delta))$.
\end{center} 
If we use the same argument 
as above, we see that the leading coefficient of $f'(x)$ belongs to $I^{(\alpha+1)}.$ Continuing this procedure, we get $a_{i}\in I^{(\alpha+1)}$ for each $i\leq n$. Thus, $f(x)\in I^{(\alpha+1)}((x^{-1};\delta)).$ 
Therefore, we obtain that
\begin{center}
$0=I^{(0)}((x^{-1};\delta))\subseteq I^{(1)}((x^{-1};\delta)) \subseteq \ldots \subseteq I^{(\alpha)}((x^{-1};\delta))\subseteq \ldots \subseteq I((x^{-1};\delta))$
\end{center}    
is the upper left annihilator series of $I((x^{-1};\delta))$ as desired.
\end{proof}
\noindent
Given $f(x)=\sum_{i=-\infty}^{n}a_{i}x^{i}\in R((x^{-1};\delta))$ with $a_{i}\in R$ for all $i\leq n$, we write $\delta^{j}(f(x))=\sum_{i=-\infty}^{n}\delta^{j}(a_{i})x^{i}$ where $j$ is a non-negative integer. And as noted in \cite{Hong4}, it can be seen that  the formula $xf(x)-f(x)x=\delta(f(x))$ holds for any $f(x)\in R((x^{-1};\delta))$ and recursively the following formula can be obtained 
\begin{center}
$\delta^{j}(f(x))=\sum_{i=0}^{j}(-1)^{j-i}\binom{j}{i}x^{i}(f(x))x^{j-i}.$
\end{center}
\begin{theorem}\label{MainTheo1}
Let $R$ be a left or right Noetherian ring with a derivation $\delta$. Then 
\begin{center} 
$\mathcal{I}_{l}(R((x^{-1};\delta)))=\mathcal{I}_{l,\delta}(R)((x^{-1};\delta))$,
\end{center}
where $\mathcal{I}_{l,\delta}(R)=\{a \in R : \sum_{j=0}^{\infty}\delta^{j}(a)R \text{ is left T-nilpotent}\}$.
\end{theorem}

\begin{proof}
Let $f(x)=\sum_{i=-\infty}^{n}a_{i}x^{i}\in \mathcal{I}_{l}(R((x^{-1};\delta)))$, where
$a_{i}\in R$ for all $i\leq n$. We want to show that 
$a_{i}\in \mathcal{I}_{l,\delta}(R)$ for all $i\leq n$.
Fix a sequence of elements $r_{1}, r_{2},\ldots \in R$ and a sequence of non-negative integers $i_1, i_2, \ldots $ and also set $s_{k}:=\delta^{i_1}(a_n)r_1\delta^{i_2}(a_n)r_2\ldots \delta^{i_k}(a_n)r_k$, where $k$ is a non-negative integer. By the assumption and \cite[Lemma 2.2]{Hong4},  we have that each of the elements 
\begin{center}
$\delta^{i_1}(f(x))r_1, \delta^{i_2}(f(x))r_2,\ldots $ 
\end{center}
belongs to a left $T$-nilpotent ideal. Hence, there exists some index $k$ such that 
\begin{center}
$\delta^{i_1}(f(x))r_1\delta^{i_2}(f(x))r_2\ldots \delta^{i_k}(f(x))r_k=0$.
\end{center}
If we expand this product, then we see that the leading coefficient is $s_{k}=0$. This means that $\sum_{j=0}^{\infty}\delta^{j}(a_n)R$ is left T-nilpotent and we get 
$a_{n}\in \mathcal{I}_{l,\delta}(R)$.
By Lemma \ref{extofTnilpotence}, we know that $a_{n}R((x^{-1};\delta))$ is a subset of the left T-nilpotent ideal $\sum_{j=0}^{\infty}\delta^{j}(a_{n})R((x^{-1};\delta))$ of $R((x^{-1};\delta))$. 
Therefore, $a_nx^{n}\in \mathcal{I}_{l}(R((x^{-1};\delta)))$. Set $f'(x)=f(x)-a_nx^{n}$. Then we have 
$f'(x)\in \mathcal{I}_{l}(R((x^{-1};\delta)))$. Thus, the leading coefficient of $f'(x)$, namely $a_{n-1}$, belongs to  $\mathcal{I}_{l,\delta}(R)$. And if we apply the same procedure, then we obtain $a_{i}\in \mathcal{I}_{l,\delta}(R)$ for each $i\leq n$. So, $f(x)=\sum_{i=-\infty}^{n}a_{i}x^{i} \in \mathcal{I}_{l,\delta}(R)((x^{-1};\delta))$ as desired.
 
Conversely, let $f(x)=\sum_{i=-\infty}^{n}a_{i}x^{i}\in \mathcal{I}_{l,\delta}(R)((x^{-1};\delta))$. We want to show that $f(x)=\sum_{i=-\infty}^{n}a_{i}x^{i}\in \mathcal{I}_{l}(R((x^{-1};\delta)))$.
It is a well-known fact that if $R$ is left or right Noetherian, then the prime radical $P(R)$ of $R$ is nilpotent. Without loss of generality, we may assume that the nilpotency index of $P(R)$ is $k$. Let $g_1(x), g_2(x), \ldots \in R((x^{-1},\delta))$. Consider the product 
\begin{center}
$f(x)g_1(x)f(x)g_2(x)\ldots f(x)g_{k}(x)$.
\end{center}
Each coefficient of this product is a $\mathbb{Z}$-linear combination of elements of the form
\begin{center}
$s_{k}:=\delta^{i_1}(a_{i_1})\delta^{i'_1}(b_1)\delta^{i_2}(a_{i_2})\delta^{i'_2}(b_2)\ldots \delta^{i_{k}}(a_{i_{k}})\delta^{i'_{k}}(b_{k}),$
\end{center}  
where $a_{i_{j}}$ is any coefficient of $f(x)$, $b_{j}$ is any coefficient of $g_{j}(x)$ and $1\leq j\leq k$.
Since $\sum_{j=0}^{\infty}\delta^{j}(a_{i})R$ is left T-nilpotent for each $i\leq n$, by Proposition \ref{Prop2.2} we have that $\sum_{j=0}^{\infty}\delta^{j}(a_{i})R\subseteq P(R)$ for each $i\leq n$.
We also know that $P(R)$ is nilpotent of index $k$, hence all the coefficients $s_{k}$ are equal to zero. Then we get the result. 
\end{proof}

In \cite[section 5]{Hong1}, the higher left T-nilpotent radideals are defined as follows:
Set $\mathcal{I}_{l}^{(0)}=0$. Let $\alpha$ be a given ordinal. If $\alpha$ is 
the successor of $\beta$, set 
\begin{center}
$\mathcal{I}_{l}^{(\alpha)}(R)=\{a\in R:\ a+\mathcal{I}_{l}^{(\beta)}(R)\in \mathcal{I}_{l}(R/\mathcal{I}_{l}^{(\beta)}(R))\}$.
\end{center}
If $\alpha$ is a limit ordinal, then we define
\begin{center}
$\mathcal{I}_{l}^{(\alpha)}(R)=\underset{\beta<\alpha}{\cup}\mathcal{I}_{l}^{(\beta)}(R)$.
\end{center}
As mentioned in \cite{Hong1}, one can define the prime radical of 
a ring $R$ alternatively as the limit of the left T-nilpotent radideals.
Now, our aim is to generalize Theorem \ref{MainTheo1} for higher left T-nilpotent radideals 
by using transfinite induction. 
Hence, we obtain a new characterization for the prime radical of pseudo-differential operator rings $P(R((x^{-1};\delta)))$. 
\begin{proposition}\label{higherversion}
Let $R$ be a left or right Noetherian ring with a derivation $\delta$. Then the higher left T-nilpotent radideals satisfy 
\begin{center}
$\mathcal{I}_{l}^{(\alpha)}(R((x^{-1};\delta)))= \mathcal{I}_{l,\delta}^{(\alpha)}(R)((x^{-1};\delta))$,
\end{center}
for any ordinal $\alpha$.
\end{proposition}

\begin{proof}
We will use transfinite induction to prove the statement. For $\alpha=1$, the result is clear. Assume that the result is true for every ordinal $\beta <\alpha$.
If $\alpha$ is not a limit ordinal, then $\alpha$ is a successor of some ordinal $\beta$ 
and by the assumption, we have 
$\mathcal{I}_{l}^{(\beta)}(R((x^{-1};\delta)))=\mathcal{I}_{l,\delta}^{(\beta)}(R)((x^{-1};\delta))$. 
We consider the natural surjection 
\begin{center}
$R((x^{-1};\delta))\rightarrow R((x^{-1};\delta))/\mathcal{I}_{l,\delta}^{(\alpha)}(R)((x^{-1};\delta))$
\end{center}
and the natural isomorphism
\begin{center}
$R((x^{-1};\delta))/\mathcal{I}_{l,\delta}^{(\alpha)}(R)((x^{-1};\delta))\cong (R/\mathcal{I}_{l,\delta}^{(\alpha)}(R))((x^{-1};\bar{\delta}))$,  
\end{center}
where $\bar{\delta}$ is the derivation of the factor ring $R/\mathcal{I}_{l,\delta}^{(\alpha)}(R)$ induced by $\delta$. 
By Theorem \ref{MainTheo1}, we have that the coefficients of the elements of  $\mathcal{I}_{l}\bigg(\big(R/\mathcal{I}_{l,\delta}^{(\beta)}(R)\big)((x^{-1};\bar{\delta}))\bigg)$ are determined by the the ideal $\mathcal{I}_{l,\bar{\delta}}(R/\mathcal{I}_{l,\delta}^{(\beta)}(R))$. By using the natural isomorphism and the natural surjection, we get the result.
If $\alpha$ is a limit ordinal, then by Theorem \ref{MainTheo1} we have 
\begin{center}
$\mathcal{I}_{l}^{(\alpha)}(R((x^{-1};\delta)))=\underset{\beta<\alpha}{\cup}\mathcal{I}_{l}^{(\beta)}(R((x^{-1};\delta)))=\bigg(\underset{\beta<\alpha}{\cup}\mathcal{I}_{l,\delta}^{(\beta)}(R)\bigg)((x^{-1};\delta)).$
\end{center} 
Therefore, we can take 
\begin{center}
$\mathcal{I}_{l,\delta}^{(\alpha)}(R)=\bigg(\underset{\beta<\alpha}{\cup}\mathcal{I}_{l,\delta}^{(\beta)}(R)\bigg)$.
\end{center}
\end{proof}
\noindent
The following example illustrates that the "Noetherian" condition of Theorem \ref{MainTheo1}
is essential.
\begin{example}
Consider the commutative ring 
\begin{center}
$R=\mathbb{Z}[a_1,a_2,\ldots :\ a_{i}^{i+1}=0 \ \text{for each } i\geq 1].$
\end{center}
Clearly, $a_{i}R$ is nilpotent and so left $T$-nilpotent for each $i$. Thus the left $T$-nilpotent radideal of $R$
is $\mathcal{I}_{l}(R)=(a_1,a_2,\ldots)$. Let $\delta=0.$ Then we have 
$\mathcal{I}_{l,\delta}(R)=\mathcal{I}_{l}(R)$.  
Consider $f(x)=\sum_{i=-\infty}^{1}a_{i}x^{i}\in \mathcal{I}_{l,\delta}(R)((x^{-1};\delta))).$ This series is not nilpotent.
\end{example}
\noindent
As a direct consequence of  Proposition \ref{higherversion}, we have the following: 
\begin{proposition}
Let $R$ be a left or right Noetherian ring with a derivation $\delta$. Then we have
\begin{center}
$P(R((x^{-1};\delta)))=P_{\delta}((x^{-1};\delta))$,
\end{center}
where $P_{\delta}((x^{-1};\delta))$ is the limit of the left T-nilpotent radideals $\mathcal{I}_{l,\delta}^{(\alpha)}(R)$.
\end{proposition} 

$${\bf Acknowledgements}$$
The hospitalities of Vladimir Bavula and the University of Sheffield are greatly acknowledged.

\noindent
H. Melis Tekin Akcin\\
Hacettepe University\\
Department of Mathematics\\
Beytepe/Ankara, 06800\\
e-mail: hmtekin@hacettepe.edu.tr

\end{document}